\documentclass[12pt]{article}

\usepackage{amsmath,amsthm,amsfonts,amssymb, color}

\newtheorem{theorem}{Theorem}[section]

\newtheorem{lemma}[theorem]{Lemma}
\newtheorem{definition}[theorem]{Definition}
\newtheorem{remark}[theorem]{Remark}

\def\cB{\mathcal{B}}
\def\cD{\mathcal{D}}
\def\cE{\mathcal{E}}
\def\cF{\mathcal{F}}

\def\cH{\mathcal{H}}

\def\cL{\mathcal{L}}
\def\cN{\mathcal{N}}
\def\cS{\mathcal{S}}
\def\cU{\mathcal{U}}

\def\bD{\mathbb{D}}
\def\bR{\mathbb{R}}

\begin{document}

\title{A gentle introduction to SPDEs: \\
the random field approach}

\author{Raluca M. Balan\footnote{ University of Ottawa, Department of Mathematics and Statistics, STEM Building, 150 Louis-Pasteur Private, Ottawa, Ontario, K1N 6N5, Canada. E-mail
address: rbalan@uottawa.ca}\\University of Ottawa}

\date{December 6, 2018}
\maketitle

\begin{abstract}
\noindent These notes constitute the basis for the lectures given by the author at Centre de recherches math\'ematiques (CRM) at Universit\'e de Montreal, as part of the thematic semester on ``Mathematical challenges in many-body physics and quantum information'' (September-December 2018). They are intended for researchers in mathematics who have a background (and an interest) in probability theory, but may not be familiar with the area of stochastic analysis, and in particular with stochastic partial differential equations (SPDEs). Their goal is to give a brief and concise introduction to the study of SPDEs using the random field approach, an area which has been expanding rapidly in the last 30 years, after the publication of John Walsh's lecture notes \cite{walsh86}. These notes do not  survey all the developments in this area, but have the rather modest goal of introducing the readers to the basic ideas, and (hopefully) spark their interest to learn more about this subject.
\end{abstract}

%60H15 s.p.d.e.
%37H15 Multiplicative ergodic theory, Lyapunov exponents

\vspace{1mm}


\newpage

\tableofcontents

\section{Introduction}

Mathematicians with a taste for history often wonder what were the most influential ideas in mathematics in the 20th century. In these notes, we present arguments to support the claim that It\^o's construction of the stochastic integral with respect to the Brownian motion should probably be on the list of such important ideas. To justify this claim in just one sentence, we say that It\^o's construction lead to the development of stochastic analysis, a field which has connections with many other areas in mathematics (for instance partial differential equations via Feynman-Kac formulae, functional analysis via Malliavin calculus), is used in a variety of applications (most notably in finance), and now includes as a youngest offspring the area of stochastic partial differential equations (SPDEs), which is of interest to some theoretical physicists.

These notes are organized as follows.

In Section \ref{section-ito}, we review It\^o's construction of the stochastic integral with respect to Brownian motion, and we show two different methods for solving some simple stochastic differential equation (SDE), using Picard's iterations, respectively series expansions. A very readable account of stochastic analysis with respect to Brownian motion can be found in \cite{kuo}.

In Section \ref{section-Walsh}, we explain Walsh's method introduced in his 1984 lectures notes \cite{walsh86} at the Saint Flour summer school, for extending It\^o's construction to higher-dimensions, i.e. to space-time white noise. Random fields (or multi-parameter processes) such as Brownian sheet which are at the core of this approach  were the focus of many investigations in the 1980's. There exist other approaches to SPDEs in the literature: the infinite-dimensional approach of Da Prato and Zabczyk (see \cite{daprato-zabczyk92}) and the analytic approach of Krylov (see \cite{krylov99}). Each of these approaches has been fruitful in its own ways. In these notes, we focus only on the random field approach, which we believe is closer to It\^o's original ideas, and has the advantage that allows us to investigate the probabilistic behavior of the solution, simultaneously in time and space. Expository lectures introducing the random field approach can be found in \cite{dalang-course} and \cite{khoshnevisan14}, whereas \cite{dalang-quer11} gives a comparison between the random field and the infinite-dimensional approach.  On the other hand, the random field approach uses tools which deviate significantly from those used for classical PDEs.
In a landmark paper \cite{dalang99},  Dalang provided the tools for analyzing a large class of SPDEs with very general noise.
We illustrate how this method works for a simple SPDE, in the linear and non-linear case. We conclude this section with the Parabolic Anderson model with space-time white noise, which is of great interest to physicists. We show that for this model, the existence of the solution can be proved using a series expansion.

In Section \ref{section-fBm}, we give some historical remarks related to the fractional Brownian motion (fBm) with Hurst index $H$ and we discuss some of the difficulties encountered in the stochastic analysis with respect to this process. We focus only on the case $H>1/2$. We present some integration techniques, which are essentially the same as for the It\^o's integral in the case of deterministic integrands, and are borrowed from Malliavin calculus in the case of random integrands. For this material, we drew heavily from Nualart's influential survey article \cite{nualart03} on fBm.
We mention an open problem related to the existence of solution to a SDE driven by fBm (using Malliavin calculus). We examine a simple SDE for which the existence of the solution can be proved using the method of series expansions (which seems to be new in the literature).

Finally, in Section \ref{section-SPDE}, we review some recent results related to SPDEs with ``colored'' noise in space and time, using the random field approach. The idea is to develop a framework which combines and extends Dalang's theory with the recent developments in the stochastic analysis for fBm. This idea appeared in \cite{BT-ALEA}, being inspired by \cite{nualart-ouknine}, which introduced a noise that was essentially a fBm in time and a Brownian motion in space. Without going into technical details,
we give a list of the results obtained in this framework for the heat and wave equations in the last 10 years, which illustrate the dynamical interplay between the regularity of the noise and various properties of the solution (such as intermittency and Feyman-Kac representations). For the heat equation, these results were obtained primarily by the school of David Nualart and Yaozhong Hu at University of Kansas. For the wave equation, the list includes results obtained by the author of these notes, in collaboration with Daniel Conus, Lluis Quer-Sardaynons, Jian Song and Ciprian Tudor.

\section{It\^o integral and some simple SDEs}
\label{section-ito}

In this section, we introduce the It\^o integral with respect to the Brownian motion, and we solve some simple SDEs driven by Brownian motion.

\subsection{It\^o integral: a breakthrough idea}

Let $(B_t)_{t \geq 0}$ be a Brownian motion defined on a complete probability space $(\Omega,\cF,P)$. The map $t \mapsto B_{t}(\omega)$ is not differentiable, for any $\omega \in \Omega$. So how can we define the integral $B(\varphi)=\int_{0}^{T} \varphi(t)dB_t$ for a deterministic function $\varphi:[0,T] \to \bR$, for fixed $T>0$?

The construction starts from the important remark that $$E(B_t B_s)=t \wedge s=\langle 1_{(0,t]}, 1_{(0,s]} \rangle_{L^2([0,T])}.$$

We define $B(1_{(0,t]})=B_t$ and we extend this definition to all simple functions, i.e. linear combinations of functions of the form $1_{(0,t]}$. The map $1_{(0,t]} \mapsto B_t$ is an isometry from $L^2([0,T])$ to $L^2(\Omega)$.
Since simple functions are dense in $L^2([0,T])$, this map can be extended to $L^2([0,T])$. We define in this way the isometry $B:L^2([0,T]) \to L^2(\Omega)$:
$$E|B(\varphi)|^2 =\int_0^T |\varphi(t)|^2 dt, \quad \mbox{for all} \ \varphi \in L^2([0,T]).$$

How can we extend this construction to a random integrand $X=\{X(t)\}_{t \geq 0}$? For this, we use the idea introduced by It\^o's in \cite{ito44}. The idea was extended later to martingales with continuous (or even c\`adl\`ag) sample paths: the integrator process (in this case the Brownian motion $(B_t)_{t \geq 0}$) is a martingale, and the integral $I^B(X)=\{I_t^B(X)\}_{t \geq 0}$ is also as a martingale.

The simplest case is when the integrand $X=(X_{t})_{t\geq 0}$ is of the form:
\begin{equation}
\label{elementary}
X(\omega,t)=Y(\omega)1_{(a,b]}(t), \quad \mbox{where $0<a<b$ and $Y$ is $\cF_a^B$-measurable}.
\end{equation}
Here $\cF_t^B=\sigma(\{B_s;s \leq t\}) \wedge \cN$, where $\cN=\{A \in \cF; P(A)=0 \ \mbox{or} \ P(A)=1\}$ is the $\sigma$-field of $P$-null sets.
In this case, we define $I_t^B(X)=Y (B_{t \wedge b}-B_{t \wedge a})$. It can be proved that
$I^B(X)$ is a martingale with respect to the filtration $(\cF_t^B)_t$ and
\begin{equation}
\label{isometry}
E|I_t^B(X)|^2=E\int_0^t |X(s)|^2 ds.
\end{equation}

Let $\cE$ be the set of {\em simple processes} on $\bR_{+}$, i.e. linear combinations of processes of form \eqref{elementary}. Let $\cL^B$ be the set of measurable $(\cF_t^B)_t$-adapted processes $X$ such that
$\|X\|_t^2:=E\int_0^t |X(s)|^2 ds<\infty$ for any $t>0$. We endow $\cL^B$ with the norm:
$$\|X\|=\sum_{k \geq 1}\frac{1\wedge \|X\|_k}{2^k}.$$
Recall that a process $X=(X_t)_{t \geq 0}$ is {\em measurable} if the map $(\omega,t) \mapsto X_t(\omega)$ is $\cF \times \cB(\bR)$-measurable on $\Omega \times \bR_{+}$, and  {\em $(\cF_t)_t$-adapted} if $X_t$ is $\cF_t$-measurable for any $t \geq 0$ (for a given filtration $(\cF_t)_t$ of sub-$\sigma$-fields of $\cF$).

It can be proved that $\cE$ is dense in $\cL^B$. For any $X \in \cL^B$ and $t >0$, we can define an element $I_t^B(X)$ in $L^2(\Omega)$ by approximating $X$ with a sequence $(X_n)_n$ in $\cE$. The process $\{I_t^B(X)\}_t$ is a martingale and relation \eqref{isometry} continues to hold for any $X \in \cL^B$. We use the notation $$I_t^B(X)=\int_0^t X(s)dB(s)$$ and we say that $I_t^B(X)$ is the {\em stochastic integral} (or {\em It\^o integral}) of $X$ with respect to $B$.

An important property of this integral is the following result, called {\em It\^o formula}.

\begin{theorem}[Theorem 7.3.1 of \cite{kuo}]
\label{Ito-formula}
If $F:[0,\infty) \times \bR \to \bR$ is a function which is continuously differentiable in $t$ and twice continuously differentiable in $x$, then
$$F(t,B_t)-F(0,0)=\int_0^t \frac{\partial F}{\partial t}(s,B_s)ds+\int_0^t \frac{\partial F}{\partial x}(s,B_s)dB(s)+\frac{1}{2}\int_0^t \frac{\partial^2 F}{\partial x^2}(s,B_s)ds.$$
\end{theorem}

\subsection{Picard's iterations: a story on how to solve a simple SDE}

For the purpose of the connection with some SPDEs which we will see later in these notes, we consider the following SDE:
\begin{equation}
\label{SDE}
dX(t)=\sigma(X(t))dB(t), \quad X(0)=0.
\end{equation}
Here $\sigma$ is a globally Lipschitz function, that is there exists a constant $C_{\sigma}>0$ such that
\begin{equation}
\label{Lip}
|\sigma(x)-\sigma(y)| \leq C_{\sigma}|x-y| \quad \mbox{for all} \ x,y \in \bR.
\end{equation}
In particular, $|\sigma(x)|^2 \leq D_{\sigma}^2 (1+|x|^2)$ for all $x \in \bR$, where $D_{\sigma}^2=2\max\{C_{\sigma}^2,|\sigma(0)|^2\}$.
By definition, the {\em solution} to \eqref{SDE} satisfies
$$X(t)=\int_0^t \sigma(X(s))dB(s),$$
provided that the stochastic integral is well-defined. A sufficient condition for this integral to be well-defined is that $X$ is measurable and $(\cF_t^B)_t$-adapted and
$$K_T:=\sup_{t \in [0,T]}E|X(t)|^2 <\infty \quad \mbox{for all} \ T>0.$$
This is because in this case, the process $\{\sigma(X(t))\}_{t\geq 0}$ is measurable and $(\cF_t^B)_t$-adapted, and
$$E\int_0^t |\sigma(X(s))|^2 ds \leq D_{\sigma}^2 \int_0^t \Big(1+E|X(s)|^2\Big) ds\leq D_{\sigma}^2(t+K_t)<\infty.$$

Two processes $(X_t)_{t\geq 0}$ and $(Y_t)_{t\geq 0}$ are {\em modifications} of each other if $P(X_t=Y_t)=1$ for all $t \geq 0$.
We have the following result.

\begin{theorem}
\label{sol-SDE-th}
There exists a solution to equation \eqref{SDE} and this solution is unique (up to a modification).
\end{theorem}

\begin{proof}
The existence part is proved using a Picard's iterations scheme. For any $t \geq 0$, let $X_0(t)=0$ and set
\begin{equation}
\label{def-Picard}
X_{n+1}(t)=\int_0^t \sigma(X_n(s))dB(s),
\end{equation}
for all $n\geq 0$. The following property is proved by induction on $n \geq 0$:
\begin{equation}
\label{propertyP}
\tag{P}
\left\{
\begin{array}{rcl}
& (i)  &  \ \displaystyle X_n(t) \ \mbox{is well-defined for any} \ t \geq 0  \\[1ex]
& (ii)  &  \displaystyle \sup_{t \in [0,T]}E|X_n(t)|^{2}<\infty \ \mbox{for any $T>0$}, \\ [1ex]
& (iii) & \displaystyle t \mapsto X_n(t) \ \mbox{is $L^2(\Omega)$-continuous} \\ [1ex]
& (iv) & \displaystyle X_n(t) \ \mbox{is $\cF_t^B$-measurable for any $t\geq 0$}.
\end{array} \right.
\end{equation}

In particular, {\em (iii)} implies that $\{X_n(t)\}_{t\geq 0}$ has a measurable modification $\{\widetilde{X}_n(t)\}_{t\geq 0}$, i.e. $X_n(t)=\widetilde{X}_n(t)$ a.s. for any $t\geq 0$. In fact, we work with this modification for defining $X_{n+1}(t)$, but to simplify the notation we denote it also by $\{X_n(t)\}_{t\geq 0}$.

Fix $T>0$. We show that $\{X_n(t)\}_n$ converges in $L^2(\Omega)$, uniformly in $t \in [0,T]$. To see this, let
$H_{n+1}(t)=E|X_{n+1}(t)-X_n(t)|^2$. By the isometry property \eqref{isometry} of the stochastic integral and the Lipschitz property \eqref{Lip} of $\sigma$,
$$H_{n+1}(t)=E\int_0^t |\sigma(X_n(s))-X_{n-1}(s)|^2 ds \leq C_{\sigma}^2 \int_0^t H_n(s) ds.$$

By Lemma \ref{gronwall} below, it follows that $\sum_{n\geq 1}\sup_{t \leq T}H_n^{1/2}(t)<\infty$. We denote by $\|\cdot\|_2$ the norm in $L^2(\Omega)$. Then
$$\sup_{t\leq T}\|X_{m}(t)-X_{n}(t)\|_2 \leq \sum_{k=n+1}^m \sup_{t\leq T}\|X_k(t)-X_{k-1}(t)\|_2=\sum_{k=n+1}^m \sup_{t\leq T}H_k^{1/2}(t)\to 0,$$
as $n,m \to \infty$. This shows that $\{X_n(t)\}_{n\geq 1}$ is a Cauchy sequence in $L^2(\Omega)$, uniformly in $t \in [0,T]$. If we denote by $X(t)$ the limit of this sequence, then $\{X(t)\}_{t \geq 0}$ is the solution to equation \eqref{SDE}; to see this, simply take the limit (in $L^2(\Omega)$) as $n\to \infty$ in \eqref{def-Picard}.

The solution is unique (up to a modification). To see this, note that if $\{Y(t)\}_{t\geq 0}$ is another solution of \eqref{SDE}, then
$$E|X(t)-Y(t)|^2=E\int_0^t |\sigma(X(s))-\sigma(Y(s))|^2 ds \leq C_{\sigma}^2 \int_0^t E|X(s)-Y(s)|^2 ds,$$
and hence, by Lemma \ref{gronwall}, $E|X(t)-Y(t)|^2=0$ for all $t \geq 0$. This proves that $X(t)=Y(t)$ a.s. for all $t \geq 0$.
\end{proof}

In the previous proof, we used the following result.

\begin{lemma}[Lemma 10.2.4 of \cite{kuo}] (Gronwall Lemma)
\label{gronwall}
Let $f_n:[0,T] \to [0,\infty)$ be such that
$$f_{n+1}(t) \leq \beta \int_0^t f_n(s)ds,$$
for all $t \in [0,T]$ and $n\geq 1$. Suppose that $f_1(t) \leq M$ for all $t \in [0,T]$. Then
$$f_n(t) \leq M \beta^{n-1} \frac{t^{n-1}}{(n-1)!},$$
for all $t \in [0,T]$ and $n\geq 1$. In particular, $\sum_{n\geq 1} \sup_{t \leq T}f_n^{1/2}(t)<\infty$.
\end{lemma}
\subsection{Particular case: SDE with multiplicative noise}

We consider now the particular case $\sigma(x)=x$. More precisely, we consider the equation: \begin{equation}
\label{SDE-multip}
dX(t)=X(t)dB(t), \quad X(0)=1,
\end{equation}
whose solution is a measurable $(\cF_t^B)_t$-adapted process $\{X(t)\}_{t \geq 0}$ which satisfies:
\begin{equation}
\label{SDE-multip-int}
X(t)=1+\int_0^t X(s)dB(s).
\end{equation}

The existence of the solution to this equation can be proved using other two different methods, which have the advantage (over the method of Picard's iterations) that yield the explicit form of the solution.

\begin{theorem}
\label{exist-sol-BM}
The solution to equation \eqref{SDE-multip} is the geometric Brownian motion: $$X(t)=e^{B_t-t/2}.$$
\end{theorem}

\begin{proof} The standard method is to apply It\^o formula (Theorem \ref{Ito-formula}) to the function $F(t,x)=e^{x-t/2}$. We present here another method, which has the advantage that can be applied to equations driven by a more general Gaussian noise than Brownian motion.
Intuitively, this method consists in writing $X(s)=1+\int_0^s X(r)dB(r)$ in \eqref{SDE-multip-int}, and iterating this procedure. It can be shown that, if it exists, the solution to \eqref{SDE-multip} can be written as the series
\begin{equation}
\label{series-sol}
X(t)=1+\sum_{n\geq 1} J_n^B(t),
\end{equation}
where $J_n^B(t)$ is the iterated integral:
$$J_n^B(t)=\int_0^t \left( \int_{0}^{t_n} \ldots \left(\int_0^{t_2} 1 dB(t_1)\right) \ldots dB(t_{n-1}) \right) dB(t_n).$$

Therefore, it is enough to show that the series \eqref{series-sol} converges in $L^2(\Omega)$.
By the symmetry of the integrand, $J_n^{B}(t)=\frac{1}{n!}I_n^B(t)$,
where $I_n^B(t)$ is the multiple integral of order $n$:
$$I_n^{B}(t)=\int_{[0,t]^n}1 dB(t_1) \ldots dB(t_n).$$
The integrals $I_n^B(t)$ can be calculated using Hermite polynomials $H_n(x)$. For $n=2$ and $n=3$, this can be seen using iterated integrals, as follows. Recall that
 $H_2(x)=x^2-1$, $H_3(x)=x^3-3x$. By applying It\^o formula (Theorem \ref{Ito-formula}) to the functions $F(x)=x^2$, respectively $F(t,x)=\frac{x^3}{3}-tx$, we see that
$$I_2^B(t)=2\int_0^t B_sdB(s)=B_t^2-t=tH_2\left(\frac{B_t}{t^{1/2}}\right).$$
$$I_3^B(t)=6 \int_0^t \left(\int_0^s B_r dB(r)\right) dB(s)=3 \int_0^t (B_s^2-s) dB(s)=
B_t^3-3t B_t =t^{3/2}H_3\left(\frac{B_t}{t^{1/2}} \right).$$

In general for higher order $n$, let $H_n$ be the Hermite polynomial of order $n \geq 1$, i.e.
$$H_n(x)=(-1)^n e^{x^2/2} \frac{d^n}{dx^n}e^{-x^2/2}.$$
Using Theorem \ref{int-Hermite} below with $W=B$, $\cH=L^2([0,T])$ and $h=1_{[0,t]}$, we see that
$$I_n^B(t)=t^{n/2}H_n\left(\frac{B_t}{t^{1/2}} \right).$$
Hence,
$$X(t)=1+\sum_{n\geq 1}\frac{1}{n!}I_n^B(t)=1+\sum_{n\geq 1}\frac{1}{n!} t^{n/2} H_n\left(\frac{B_t}{t^{1/2}} \right)=e^{B_t-t/2},$$
where for the last equality, we used the following property of Hermite polynomials:
\begin{equation}
\label{Hermite}
e^{tx-t^2/2}=1+\sum_{n\geq 1}\frac{1}{n!}t^n H_n(x).
\end{equation}
\end{proof}

In the previous proof, we used the following result:

\begin{theorem}[Theorem 2.7.7 of \cite{nourdin-pecatti12}]
\label{int-Hermite}
Let $W=\{W(h)\}_{h \in \cH}$ be an isonormal Gaussian process corresponding to the Hilbert space $\cH$, i.e. a zero-mean Gaussian proces with covariance $$E[W(g)W(h)]=\langle g,h\rangle_{\cH}.$$ If $I_n$ is the multiple integral of order $n$ with respect to $W$, then for any $h \in \cH$,
$$I_n(h^{\otimes n})=\|h\|_{\cH}^nH_n \left(\frac{W(h)}{\|h\|_{\cH}} \right).$$
\end{theorem}

We are interested in the asymptotic behavior of $X(t)$ for $t$ large. By the strong law of large numbers, $B_n/n \to 0$ a.s., and hence $X(n)=e^{n(B_n/n-1/2)} \to 0$ a.s. So with probability $1$, $X(n)$ is small when $n$ is large. But {\em on average}, $X(n)$ may not be so small if $n$ is large. More precisely, $E[X(t)^p]=\exp\{\frac{p(p-1)}{2}t\}$ for any $p>0$, using the fact that $E(e^{\lambda Z})=e^{\lambda^2\sigma^2/2 }$ for any $\lambda>0$, if $Z$ is a $N(0,\sigma^2)$ random variable. Hence for any $p>0$,
$$\lambda_p:=\lim_{t \to \infty}\frac{1}{t} \log E[X(t)^p]=\frac{p(p-1)}{2}.$$
$\lambda_p$ is called the {\em Lyapunov exponent} of order $p$ of the process $\{X(t)\}_{t \geq 0}$ and is related to the concept of intermittency from physics, which intuitively speaks about the average growth of $\{X(t)\}_{t \geq 0}$ on an exponential scale. Rigorously, we have the following definition:

\begin{definition}
\label{def-intermit}
{\rm The {\em upper Lyapunov exponent of order $p>0$} of a process $\{X(t)\}_{t \geq 0}$ is:
$$\overline{\lambda}_p:=\limsup_{t \to \infty}\frac{1}{t} \log E|X(t)|^p.$$
A process $\{X(t)\}_{t \geq 0}$ is {\em weakly intermittent} if $$\overline{\lambda}_2>0 \quad \mbox{and} \quad \overline{\lambda}_p<\infty \quad \mbox{for all $p\geq 2$},$$
and is
{\em fully intermittent} if the function $p \mapsto \overline{\lambda}_p/p$ is strictly increasing on $(0,\infty)$.
}
\end{definition}

\section{SPDEs with space-time white noise}
\label{section-Walsh}

In this section, we introduce Walsh's random field approach for solving SPDEs with space-time white noise.

\subsection{Walsh' approach: the case of space-time white noise}

Let $\{W(A);A \in \cB_{b}(\bR_{+} \times \bR^d)\}$ be a {\em space-time white noise} defined on a complete probability space $(\Omega,\cF,P)$. This is a zero-mean Gausssian process
with covariance
$$E[W(A)W(B)]={\rm Leb}(A \cap B),$$
where $\cB_{b}(\bR_{+} \times \bR^d)$ is the set of bounded Borel sets in $\bR_{+} \times \bR^d$, and ${\rm Leb}$ denotes the Lebesgue measure. The multi-parameter process $\{W(t,x)=W([0,t] \times [0,x]);t \geq 0, x \in \bR^d\}$ is called {\em Brownian sheet}. For any $x \in \bR^d$ fixed, $\{W(\cdot,x)\}_{t\geq 0}$ is a Brownian motion with variance ${\rm Leb}([0,x])$.

We proceed now with the construction of the stochastic integral with respect to the space-time white noise. We consider first the case of deterministic integrands. If $\varphi=1_A$ for some $A \in \cB_b(\bR_{+} \times \bR^d)$, we let
$I^{W}(\varphi)=W(A)$.
We extend this definition to linear combinations of indicator functions of this form. For such functions,
$$E[I^{W}(\varphi)I^{W}(\psi)]=\int_0^{\infty}\int_{\bR^d}\varphi(t,x)\psi(t,x)dtdx=\langle \varphi,\psi \rangle_{L^2(\bR^+ \times \bR^d)}.$$
This shows that $\varphi \mapsto I^{W}(\varphi) \in L^2(\Omega)$ is an isometry, which we extend to $L^2(\bR_{+} \times \bR^d)$. For any $\varphi \in L^2(\bR_{+} \times \bR^d)$, we denote
$$I^{W}(\varphi)=\int_{0}^{\infty}\int_{\bR^d}\varphi(t,x)W(dt,dx).$$

Next, we treat the case of random integrands. For this, let $\cF_{t}^W=\sigma(\{W_s(A); 0\leq s \leq t, A \in \cB_{b}(\bR^d)\}) \wedge \cN$, where $W_t(A)=W([0,t] \times A)$. We use the same idea as in the case of the Brownian motion.
For this we need the following definition.

\begin{definition}
{\rm A collection $\{M_t(A);t\geq 0, A \in \cB_b(\bR^d)\}$ of square-integrable random variables defined on a probability space $(\Omega,\cF,P)$ is a {\em martingale measure} (with respect to a filtration $(\cF_t)_t$) if:\\
{\em (a)} for any $A \in \cB_b(\bR^d)$ fixed, $\{M_t(A)\}_{t\geq 0}$ is a martingale with $M_0(A)=0$; \\
{\em (b)} for any $t>0$ fixed, $\{M_t(A)\}_{A \in \cB_b(\bR^d)}$ is a $\sigma$-finite $L^2(\Omega)$-valued signed measure, i.e. \\
{\em (i)} $M_t(A \cup B)=M_t(A)+M_t(B)$ a.s. for any disjoint sets $A,B \in \cB_b(\bR^d)$;\\
{\em (ii)} there exists a sequence $(E_k)_k$ in $\cB_b(\bR^d)$ with $E_k \uparrow R^d$ such that $\sup_{A \subset E_k}E|M_t(A)|^2< \infty$ for any $k$, and $E|M_t(A_n)|^2 \to 0$ for any $A_n \downarrow \emptyset$ with $A_n \subset E_k$ for all $n$, for some $k$.}
\end{definition}

The integrator $W$ is a martingale measure and the integral $I^W$ will also be a martingale measure. The simplest case is when the integrand $X$ is a random field of the form:
\begin{equation}
\label{elem-proc}
X(\omega,t,x)=Y(\omega) 1_{(a,b]}(t) 1_{H}(x),
\end{equation}
for some $0<a<b$, $H \in \cB_b(\bR^d)$ and an $\cF_a^W$-measurable bounded
random variable $Y$. In this case, we let
$$I_t^{W}(X)(A)=Y\big(W_{t \wedge b}(A \cap H)-W_{t \wedge a}(A \cap H) \big).$$
It can be proved that $\{I_t^W(X)(A);t\geq 0,A \in \cB_b(\bR^d)\}$ is a martingale measure with respect to $(\cF_t^W)_t$ and
\begin{equation}
\label{isometry2}
E|I_t^W(X)(A)|^2=E\int_0^t\int_{A}|X(s,x)|^2 dx ds.
\end{equation}

Let $\cS$ be the set of {\em simple random fields} on $\bR_{+} \times \bR^d$, i.e. linear combinations of processes of form \eqref{elem-proc}.  Let $\cL^W$ be the set of jointly measurable (measurable in $(\omega,t,x)$), $(\cF_t^W)_t$-adapted random fields $\{X(t,x)\}_{t \geq 0, x \in \bR^d}$ such that
$$[X]_t^2=E\int_0^t \int_{\bR^d}|X(s,x)|^2 dsdx<\infty \quad \mbox{for all} \quad t>0.$$
We endow $\cL^W$ with the norm
$$[X]=\sum_{k \geq 1}\frac{1\wedge [X]_k}{2^k}.$$
It can be proved that $\cS$ is dense in $\cL^W$. For any $X \in \cL^W$, $t >0$ and Borel set $A \subset \bR^d$, we can define an element $I_t^W(X)(A)$ in $L^2(\Omega)$ by approximating $X$ with a sequence $(X_n)_n$ in $\cS$. Then $\{I_t^W(X)(A);t \geq 0,A \in \cB_{b}(\bR^d)\}$ is a martingale measure and relation \eqref{isometry2} continues to hold for any $X \in \cL^W$.
This construction extends to $A=\bR^d$ and \eqref{isometry2} holds for this case too.
We use the notation $$I_t^W(X)(A)=\int_0^t \int_{A} X(s,x) W(ds,dx)$$ and we say that $I_t^W(X)$ is the {\em stochastic integral} (or {\em It\^o integral}) of $X$ with respect to $W$.

\subsection{The simplest SPDE: the linear equation}
\label{lin-eq-white}

Let $L$ be a second-order partial differential operator with constant coefficients on $\bR_{+} \times \bR^d$. We are interested primarily in the heat operator $L=\frac{\partial}{\partial t}-\frac{1}{2}\Delta$ and the wave operator $L=\frac{\partial^2}{\partial t^2}-\Delta$. Let $G$ be the fundamental solution of $L$, i.e. the solution of the equation
\begin{equation}
\label{fund-sol}
LG=\delta_0 \quad \mbox{in} \quad \cD'(\bR_{+} \times \bR^{d}),
\end{equation}
where $\cD'(\bR_+ \times \bR^d)$ is the space of distributions on $\bR_{+} \times \bR^d$.

We denote by $|\cdot|$ the Euclidean norm on $\bR^d$.
If $L$ is the heat operator,
\begin{equation}
\label{G-heat}
G(t,x)=\frac{1}{(2\pi t)^{d/2}}\exp \Big(-\frac{|x|^2}{2t}\Big).
\end{equation}
If $L$ is the wave operator,
\begin{align*}
G(t,x)&=\frac{1}{2}1_{\{|x|<t\}} \quad \mbox{if} \quad d=1\\
G(t,x)&=\frac{1}{2\pi} \frac{1}{\sqrt{t^2-|x|^2}} 1_{\{|x|<t\}}  \quad \mbox{if} \quad d=2\\
G(t,x)&=\frac{1}{4\pi} \sigma_t  \quad \mbox{if} \quad d=3,
\end{align*}
where $\sigma_t$ is the surface measure on the sphere $\{x \in \bR^3;|x|=t\}$. If $L$ is the wave operator in spatial dimension $d\geq 4$, $G(t,\cdot)$ is a distribution with compact support in $\bR^d$.

We consider the following linear SPDE:
\begin{equation}
\label{linear-spde}
Lu(t,x)=\dot{W}(t,x), \quad t>0,x \in \bR^d,
\end{equation}
with zero initial conditions.

A priori, we do not know if the maps $t\mapsto u(\omega,t,x)$ and $x\mapsto u(\omega,t,x)$ are differentiable, so in general, $Lu(\omega,t,x)$ is not well-defined.
By definition, the {\em mild solution} to \eqref{linear-spde} is given by:
\begin{equation}
\label{sol-linear}
u(t,x)=\int_0^t \int_{\bR^d}G(t-s,x-y)W(ds,dy),
\end{equation}
provided that the stochastic integral is well-defined, i.e.
\begin{equation}
\label{cond-G}
\int_0^t \int_{\bR^d}G^2(t-s,x-y)dyds<\infty.
\end{equation}
In the case of the heat and wave equations, this reduces to asking that $d=1$.

Why a process $u$ given by \eqref{sol-linear} is called a mild solution? The reason for this is a formal manipulation. Formally, we replace $\dot{W}$ in \eqref{linear-spde} by a non-random function $f$. From the classical PDE theory, we know that the solution of $Lu=f$ with zero initial conditions is
$u(t,x)=\int_0^t \int_{\bR^d}G(t-s,x-y)f(s,y)dyds$. Next, replacing back $f$ by $\dot{W}$, we argue that $\dot{W}(s,y)dsdy$ should be (formally) the same as $W(ds,dy)$.

\begin{remark}
\label{Dalang-rem}
{\rm In \cite{dalang99}, Dalang constructed a stochastic integral with respect to a more general noise, called {\em spatially homogeneous Gaussian noise}. This noise is given by a zero-mean Gaussian process $\{M(A);A \in \cB_b(\bR_{+} \times \bR^d)\}$ with covariance
\begin{equation}
\label{Dalang-cov}
E[M(A)M(B)]=\int_{\bR_{+}} \int_{\bR^d}\int_{\bR^d} 1_{A}(t,x)1_{B}(t,x) f(x-y)dx dy dt,
\end{equation}
for a non-negative function $f$ on $\bR^d$ which is non-negative definite in the sense of distributions (and may be $\infty$ at $0$). In this case, by Bochner-Schwartz theorem, $f$ is the Fourier transform of a tempered measure $\mu$ on $\bR^d$.
With this type of noise, the linear heat and wave equations have random-field solutions in any spatial dimension $d\geq 1$, provided that
\begin{equation}
\label{Dalang-cond}
\int_{\bR^d}\frac{1}{1+|\xi|^2}\mu(d\xi)<\infty.
\end{equation}
Condition \eqref{Dalang-cond} was introduced simultaneously in articles \cite{dalang99} and \cite{PZ00}, and is known in the literature as {\em Dalang's condition}. Formally, the case of the space-time white noise corresponds to $f=\delta_0$ and $\mu$ equal to the Lebesgue measure. An interesting example arising from potential theory is when $f$ is the {\em Riesz kernel} of order $\alpha$, i.e.
\begin{equation}
\label{Riesz-def}
f(x)=|x|^{-\alpha}, \quad \alpha \in (0,d).
\end{equation}
In this case, $f$ is the Fourier transform in the space $\cS'(\bR^d)$ of tempered distributions on $\bR^d$ of the measure $\mu(d\xi)=c_{\alpha,d}|\xi|^{-d+\alpha}d\xi$, where $c_{\alpha,d}>0$ is a constant depending on $\alpha$ and $d$ (see page 117 of \cite{stein70}). For this example, condition \eqref{Dalang-cond} becomes $\alpha<d \wedge 2$.
}
\end{remark}

\subsection{How complications arise: the non-linear equation}

Assume \eqref{cond-G} holds. Consider now the following non-linear equation:
\begin{equation}
\label{non-lin}
Lu(t,x)=\sigma(u(t,x))\dot{W}(t,x) \quad t>0,x \in \bR^d,
\end{equation}
with zero initial conditions, where $\sigma$ is a globally Lipschitz function.

By definition, a {\em mild solution} to \eqref{non-lin} satisfies:
$$u(t,x)=\int_0^t \int_{\bR^d}G(t-s,x-y)\sigma(u(s,y))W(ds,dy),$$
provided that the stochastic integral is well-defined. Note that a sufficient condition for this integral to be well-defined is that $u$ is measurable and $(\cF_t^W)_t$-adapted and
$$\sup_{(t,x) \in [0,T \times \bR^d]}E|u(t,x)|^2<\infty \quad \mbox{for all} \quad T>0.$$

We say that processes $\{X(t,x)\}_{t,x}$ and $\{Y(t,x)\}_{t,x}$ are {\em modifications} of each other if $P(X(t,x)=Y(t,x))=1$ for all $t\geq 0$ and $x\in \bR^d$.
The next result establishes the existence and uniqueness of this solution.

\begin{theorem}
If \eqref{cond-G} holds, then there exists a mild solution to equation \eqref{non-lin} and this solution is unique (up to a modification).
\end{theorem}

\begin{proof} As in the proof of Theorem \ref{sol-SDE-th}, we set up a Picard's iteration scheme. For any $t\geq 0$ and $x \in \bR^d$, let $u_0(t,x)=0$ and set
\begin{equation}
\label{def-Picard2}
u_{n+1}(t,x)=\int_0^t \int_{\bR^d}G(t-s,x-y)\sigma(u_n(s,y))W(ds,dy),
\end{equation}
for any $n\geq 0$. The following property is proved by induction on $n \geq 0$:
\begin{equation}
\label{propertyP}
\tag{P}
\left\{
\begin{array}{rcl}
& (i)  &  \ \displaystyle u_n(t,x) \ \mbox{is well-defined for any} \ t \geq 0,x \in \bR^d;  \\[1ex]
& (ii)  &  \displaystyle \sup_{(t,x) \in [0,T] \times \bR^d}E|u_n(t,x)|^{2}<\infty; \ \mbox{for any $T>0$}, \\ [1ex]
& (iii) & \displaystyle (t,x) \mapsto u_n(t,x) \ \mbox{is $L^2(\Omega)$-continuous}; \\ [1ex]
& (iv) & \displaystyle u_n(t,x) \ \mbox{is $\cF_t^W$-measurable for any $t\geq 0, x \in \bR^d$}.
\end{array} \right.
\end{equation}
In particular, from (iii) it follows that $\{u_n(t,x)\}_{t,x}$ has a measurable modification $\{\widetilde{u}_n(t,x)\}_{t,x}$. In fact, we work with this modification for defining $u_{n+1}(t,x)$, but to simplify the notation we denote it also by $\{u_n(t,x)\}_{t,x}$.

Fix $T>0$ and $p\geq 2$. We show that $\{u_n(t,x)\}_n$ converges in $L^p(\Omega)$ to $u(t,x)$, uniformly in $(t,x) \in [0,T] \times \bR^d$. To see this, let
$$M_n(t)=\sup_{x \in \bR^d}E|u_n(t,x)-u_{n-1}(t,x)|^2.$$
By the isometry property \eqref{isometry2} of the stochastic integral with respect to $W$ and the Lipschitz property \eqref{Lip} of $\sigma$,
\begin{align*}
E|u_{n+1}(t,x)-u_n(t,x)|^2 &=\int_0^t \int_{\bR^d}G^2(t-s,x-y) E|\sigma(u_n(s,y))-\sigma(u_{n-1}(s,y))|^2 dyds\\
& \leq C_{\sigma}^2 \int_0^t \int_{\bR^d}G^2(t-s,x-y) E|u_n(s,y)-u_{n-1}(s,y)|^2 dyds\\
&\leq C_{\sigma}^2 \int_0^t \sup_{y \in \bR^d}E|u_n(s,y)-u_{n-1}(s,y)|^2 \
\left( \int_{\bR^d}G^2(t-s,x-y)dy \right) ds \\
&= C_{\sigma}^2 \int_0^t M_n(s) g(t-s) ds,
\end{align*}
where $g(s)=\int_{\bR^d}G^2(s,y)dy$. Taking the supremum over all $x \in \bR^d$, we arrive at the following recurrence relation:
$$M_{n+1}(t)\leq C_{\sigma}^2 \int_0^t M_n(s) g(t-s) ds,$$
{\em which is not covered by the classical Gronwall lemma!} In \cite{dalang99}, Dalang developed a very nice extension of Gronwall lemma, which is suitable for tackling this problem. This is stated as Lemma \ref{Dalang-lemma} below. Using this lemma, it follows that
$\sum_{n\geq 1}\sup_{t \leq T}H_n^{1/p}(t)<\infty$. We denote by $\|\cdot\|_p$ the norm in $L^p(\Omega)$. Then
$$\sup_{(t,x) \in [0,T] \times \bR^d}\|u_m(t,x)-u_{n}(t,x)\|_p \leq \sum_{k=n+1}^{m}\|u_k(t,x)-u_{k-1}(t,x)\|_p \leq \sum_{k=n+1}^m H_k^{1/p}(t,x) \to 0,$$
as $n,m \to \infty$. This proves that $\{u_n(t,x)\}_{n}$ is a Cauchy sequence in $L^p(\Omega)$, uniformly in $(t,x)\in [0,T] \times \bR^d$.
If we denote by $u(t,x)$ the limit of this sequence, then $\{u(t,x)\}_{t\geq 0,x \in \bR^d}$ is the solution to equation \eqref{SDE}; to see this, simply take the limit (in $L^p(\Omega)$) as $n\to \infty$ in \eqref{def-Picard2}.

We now prove uniqueness. Let $\{v(t,x)\}_{t\geq 0,x \in \bR^d}$ be another solution of \eqref{SDE}, and  $M(t)=\sup_{x \in
\bR^d}E|u(t,x)-v(t,x)|^2$. Then
\begin{align*}
E|u(t,x)-v(t,x)|^2& =\int_0^t \int_{\bR^d}G^2(t-s,x-y)E |\sigma(u(s,y))-\sigma(v(s,y))|^2 dy ds \\
& \leq C_{\sigma}^2 \int_0^t \int_{\bR^d}G^2(t-s,x-y)E|u(s,y)-v(s,y)|^2 dy ds \\
&\leq C_{\sigma}^2 \int_0^t M(s)g(t-s)ds.
\end{align*}
Taking the supremum over $x \in \bR^d$, we obtain: $M(t) \leq C_{\sigma}^2 \int_0^t M(s)g(t-s)ds$, and hence, by Lemma \ref{Dalang-lemma} below, $E|u(t,x)-v(t,x)|^2=0$ for all $t \geq 0$ and $x \in \bR^d$. This proves that $u(t,x)=v(t,x)$ a.s. for all $t \geq 0$ and $x \in \bR^d$.
\end{proof}

\begin{lemma}[Lemma 15 of \cite{dalang99}]
\label{Dalang-lemma}
(Extension of Gronwall Lemma)
Let $f_n:[0,T] \to [0,\infty)$ be such that
$$f_{n+1}(t) \leq \int_0^t f_n(s)g(t-s)ds,$$
for all $t \in [0,T]$ and $n\geq 0$, for a non-negative function $g$ which is integrable on $[0,T]$. Suppose that $f_0(t) \leq M$ for all $t \in [0,T]$. Then for all $n \geq 0$ and $t \in [0,T]$,
$$f_n(t) \leq M a_n,$$
where $(a_n)_n$ is a sequence of positive numbers with the property that $\sum_{n}a_n^{1/p}<\infty$ for all $p>0$.  In particular, $\sum_{n\geq 1}\sup_{t \leq T}f_n^{1/p}(t)<\infty$ for all $p>0$. (More precisely, $a_n=G(T)^n P(S_n\leq T)$, where $G(T)=\int_0^T g(s)ds$ and $S_n=\sum_{i=1}^{n}X_i$, with $(X_i)_{i\geq 1}$ i.i.d. random variables on $[0,T]$ with density $g(s)/G(T)$.)
\end{lemma}

\subsection{Parabolic Anderson Model}

In this section, we consider the case of the heat equation in the particular case $\sigma(x)=x$, with non-vanishing initial conditions. More precisely, we look at the equation:
\begin{equation}
\label{PAM}
\frac{\partial u}{\partial t}(t,x)=\frac{1}{2}\Delta u(t,x)+u(t,x)\dot{W}(t,x) \quad t>0,x \in \bR^d
\end{equation}
with an initial condition $u(0,x)=1$. This equation is known in the literature as {\em the Parabolic Anderson Model} with space-time white noise.
By definition, the mild solution to \eqref{PAM} satisfies:
\begin{equation}
\label{multip-eq}
u(t,x)=1+\int_0^t \int_{\bR^d}G(t-s,x-y)u(s,y)W(ds,dy).
\end{equation}

This equation can be solved similarly to \eqref{SDE-multip}, but in this case we do not obtain such an explicit formula for the solution. By introducing $$u(s,y)=1+\int_0^s \int_{\bR^d}G(s-r,y-z)u(r,z)W(dr,dz)$$ on the right-hand side of \eqref{multip-eq} and iterating this procedure, we obtain the following series representation:
\begin{equation}
\label{PAM-series}
u(t,x)=1+\sum_{n\geq 1}I_n^{W}(f_n(\cdot,t,x)),
\end{equation}
where $I_n^W$ is the multiple integral of order $n$ with respect to $W$ (defined using Malliavin calculus; see \cite{nualart06}), and the function $f_n(\cdot,t,x)$ is given by:
\begin{equation}
\label{def-fn}
f_n(t_1,x_1,\ldots,t_n,x_n,t,x)=G(t-t_n,x-x_n) \ldots G(t_2-t_1,x_2-x_1)1_{\{0<t_1<\ldots<t_n<t\}}.
\end{equation}

It can be proved that the integral $I_n^W$ is well-defined on $L^2((\bR_{+} \times \bR^d)^n)$, and it is not hard to see that $f_n(\cdot,t,x) \in L^2((\bR_{+} \times \bR^d)^n)$. Moreover,
\begin{align*}
E|I_n^W(f_n(\cdot,t,x))|^2&:=\int_{\{0<t_1<\ldots<t_n<t\}}G^2(t-t_n,x-x_n) \ldots G^2(t_2-t_1,x_2-x_1)dt_1 dx_1 \ldots dt_n dx_n \\
&=(4\pi)^{-n/2} \int_{\{0<t_1<\ldots<t_n<t\}} [(t-t_n)\ldots (t_2-t_1)]^{-1/2}dt_1 \ldots dt_n\\
&=\frac{(t/2)^{n/2}}{\Gamma(n/2+1)}.
\end{align*}
Here, we used the fact that $G^2(t,x)=(4\pi t)^{-1/2}G(t/2,x)$ and $G(t,\cdot)$ is a density function on $\bR^d$. (Recall the form \eqref{G-heat} of $G$.)

Using Malliavin calculus techniques (see \cite{nualart06}), it can be proved that the elements in the series \eqref{PAM-series} are orthogonal in $L^2(\Omega)$. Therefore, we obtain an explicit calculation for the second moment of $u(t,x)$:
\begin{align*}
E|u(t,x)|^2&=1+\sum_{n\geq 1}E|I_n^W(f_n(\cdot,t,x))|^2=1+\sum_{n\geq 1}\frac{(t/2)^{n/2}}{\Gamma(n/2+1)}=2 e^{t/4} \Phi(\sqrt{t/2}),
\end{align*}
where $\Phi$ is the standard normal distribution function. The last equality is due to Lemma 2.3.4 of \cite{chen13}. Hence,
$$\lambda_2:=\lim_{t \to \infty}\frac{1}{t}\log E|u(t,x)|^2=\frac{1}{4}.$$

For moments of higher order, Bertini and Cancrini proved in \cite{bertini-cancrini95} that:
$$\lambda_p:=\lim_{t \to \infty}\frac{1}{t} \log E|u(t,x)|^p=\frac{1}{4!}p(p^2-1).$$
Hence, $u$ is fully intermittent (in the sense of Definition \ref{def-intermit}).

\section{Stochastic analysis for fBm}
\label{section-fBm}

In this section, we introduce the fBm and we offer a glimpse at the challenges of stochastic analysis with respect to fBm.

\subsection{A new arrival in this story: fractional Brownian motion}

Let $H \in (0,1)$ be arbitrary. The {\em fractional Brownian motion} (fBm) of index $H$ is
a zero-mean Gaussian process $(B_t^{H})_{t\geq 0}$ with covariance
$$E[B_t^H B_s^H]=\frac{1}{2}(t^{2H}+s^{2H}-|t-s|^{2H})=:R_{H}(t,s).$$
The parameter $H$ is called the {\em Hurst index}. If $H=1/2$, $(B_t^H)_{t\geq 0}$ is the Brownian motion.

\vspace{3mm}
{\bf Comment about the notation:} $B_t^H$ is just a notation, which is commonly used in the literature; $B_t^H$ is certainly not equal to the power $H$ of the Brownian motion $B_t$.
\vspace{3mm}

Since late 1990's, the fBm has been used increasingly in stochastic analysis as a replacement for the Brownian motion. One of its appealing features is the flexibility given by the choice of the index $H$. In 2003, David Nualart published a very nice survey \cite{nualart03} which reviews the properties of fBm and explains several methods for developing a stochastic calculus with respect to this process. This remains a landmark reference to this day. We include below some historical remarks and key properties of fBm taken from \cite{nualart03}.

The fBm was introduced by Kolmogorov in \cite{kolmogorov40}, who called it the ``Wiener spiral''. Nobody uses this name today. He proved that $R_H$ is non-negative definite, using a representation of the form:
$$R_{H}(t,s)=\int_0^{t \wedge s}K_{H}(t,r)K_H(r,s)dr,$$
for a certain kernel $K_H$, which has different forms for $H<1/2$ and $H>1/2$. From this, we deduce that the fBm can be represented as $B_t^H=\int_0^t K_{H}(t,s)dB_s$, where $(B_t)_{t\geq 0}$ is the Brownian motion.

The name ``fractional Brownian motion'' was coined by Madelbrott and Van Ness in
\cite{mandelbrot-vanness} who obtained the ``moving average representation'' of the fBm:
$$B_t^H=\frac{1}{c_H} \int_{\bR}[(t-s)_{+}^{H-1/2}-(-s)_{+}^{H-1/2}] dB(s),$$
where $c_H>0$ is a constant depending on $H$ and $(B_s)_{s \in \bR}$ is the Brownian motion on $\bR$.

The increments of fBm are not independent.
The name of the index $H$ comes from Hurst who studied in \cite{hurst51} the water run-offs of the Nile river, and concluded that since his data points were correlated, they should be regarded as increments of the fractional Brownian motion.

For any $s<t$, $B_t^H-B_s^H$ has a $N(0,(t-s)^{2H})$ distributions. Hence,
the fBm has stationary increments, and is self-similar of order $H$, i.e. for any $a>0$, the processes $(B_{at}^H)_{t \geq 0}$ and $(a^H B_t)_{t \geq 0}$ have the same distribution.

The process $(B_t^H)_{t \geq 0}$ has a modification $(\widetilde{B}_t^H)_{t \geq 0}$ whose sample paths are  H\"older continuous of order $H-\varepsilon$ for any $\varepsilon>0$. This follows by Kolmogorov's criterion, since
$$\Big(E|B_t^{H}-B_s^H|^p\Big)^{1/p} =c_p \Big(E|B_t^H-B_s^H|^2 \Big)^{1/2}=c_p|t-s|^H,$$
where $c_p=(E|Z|^p)^{1/p}$ and $Z$ has a $N(0,1)$-distribution.
In other words, $(\widetilde{B}_t^H)_{t \geq 0}$ has smoother sample paths than Brownian motion if $H>1/2$, and rougher paths if $H<1/2$.

Most importantly, the {\em fBm is not a martingale}. To see this, note that $\sum_{j=1}^{n}|B_{j/n}^H-B_{(j-1)/n}^H|^2$ converges almost surely to $0$ if $H>1/2$ and to $\infty$ if $H<1/2$. (Recall that if $(M_t)_{t\geq 0}$ is a square-integrable martingale with $M_0=0$, then $\sum_{j=1}^{n}|M_{jt/n}-M_{(j-1)t/n}|^2 \stackrel{P}{\to}\langle M \rangle_t$, where $\langle M \rangle$ is the quadratic variation of $M$.)

\subsection{Integration with respect to fBm with $H>1/2$}
\label{integral-fBm}

We fix $T>0$. The goal of this section is to define the integral $\int_{0}^T \varphi(t)dB^H(t)$.

Suppose that $H>1/2$. In this case, it can be proved that
$$R_H(t,s)=\alpha_H \int_0^t \int_0^s |u-v|^{2H-2}dudv,$$
with $\alpha_H=H(2H-1)$; see \cite{nualart03}. Hence
\begin{equation}
\label{fBm-cov}
E[B_t^H B_s^H]=\int_0^{T} \int_{0}^{T}1_{(0,t]}(u)1_{(0,s]}(v)|u-v|^{2H-2}dudv=: \langle 1_{(0,t]},1_{(0,s]} \rangle_{\cU}.
\end{equation}

Similarly to the Brownian motion case, we define $B^H(1_{(0,t]})=B_t^H$ and we extend this definition by linearity to simple functions. The map $1_{(0,t]} \mapsto B_t^H \in L^2(\Omega)$ is an isometry which can be extended to the Hilbert space $\cU$, defined as the completion of the set of simple functions with respect to the inner product
$\langle \cdot,\cdot \rangle_{\cU}$. We define in this way the isometry $B^H: \cU \mapsto L^2(\Omega)$:
$$E|B^H(\varphi)|^2=\|\varphi\|_{\cU}^2 \quad \mbox{for all} \ \varphi \in \cU.$$

It can be proved that $\cU$ contains elements $\varphi$ from the space $\cS'(\bR)$ of tempered distributions on $\bR$, whose Fourier transform $\cF \varphi$ is a function. More precisely, these elements belong to the Sobolev space $W^{-(H-1/2),2}(\bR)$ of order $-(H-1/2)$ and the inner product in $\cU$ can be expressed as:
$$\langle \varphi, \psi \rangle_{\cU}=c_H\int_{\bR}\cF \varphi(\xi) \overline{\cF \psi(\xi)}|\xi|^{1-2H}d\xi,$$
where $c_H>0$ is a constant depending on $H$ (see Proposition 4.1 of \cite{jolis10}). But $\cU$ contains several nice function spaces:
$$L^2([0,T]) \subset L^{1/H}([0,T]) \subset |\cU| \subset \cU,$$
where $|\cU|$ is the set of measurable functions $\varphi:[0,T] \to \bR$ such that
$$\|\varphi\|_{|\cU|}^2:=\alpha_H \int_{0}^T \int_0^T |\varphi(u)| |\varphi(v)||u-v|^{2H-2}dudv<\infty.$$
So if $\varphi$ is a function in one of these subspaces of $\cU$,
$$E[B^H(\varphi)B^H(\psi)]=\langle \varphi,\psi \rangle_{\cU}=\alpha_H \int_0^T \int_0^T \varphi(t) \psi(s)|t-s|^{2H-2}dtds.$$

\begin{remark}
\label{Fourier-rem}
{\rm
For the sake of a generalization which we will discuss later, note that the function $\gamma(t)=|t|^{2H-2}$ is the Riesz kernel \eqref{Riesz-def} of order
$\alpha=2-2H$ in dimension $d=1$. So, $\gamma$ is the Fourier transform (in the space $\cS'(\bR)$ of tempered distributions on $\bR)$ of the tempered measure
$\nu(d\tau)=C_{H}|\tau|^{1-2H}d\tau$, where $C_H>0$ is a constant depending on $H$.
}
\end{remark}

What about the case of random integrands? Since the fBm is not a martingale, we cannot use It\^o's theory. Instead of this, we will use {\em Malliavin calculus} (see e.g. \cite{nualart06}). Other methods exist in the literature, for instance defining a pathwise integral, which exploits the H\"older continuity of the sample paths of the fBm that we mentioned above. We will not discuss these methods here.

Note that $\{B^H(\varphi);\varphi \in \cU\}$ is an {\em isonormal Gaussian process}, i.e. a zero-mean Gaussian process with covariance given by the inner product in a Hilbert space: for any $\varphi,\psi \in \cU$,
$$E[B^H(\varphi)B^{H}(\psi)]=\langle \varphi, \psi \rangle_{\cU}.$$

The starting point of the construction of the integral is the Malliavin derivative. Let $F$ be a ``smooth'' random variable. i.e. a random variable of the form $$F=f(B^{H}(\varphi_1),\ldots,B^H(\varphi_n))$$ for some function $f \in C_{b}^{\infty}(\bR^n)$, $n \geq 1$ and $\varphi_1,\ldots,\varphi_n \in \cU$, where
$C_b^{\infty}(\bR^n)$ is the set of infinitely differentiable functions on $\bR^n$ with bounded partial derivatives. The {\em Malliavin derivative} of $F$ is defined as the following (random) element in $\cU$:
$$D F:=\sum_{i=1}^{n}\frac{\partial f}{\partial x_i}(B^{H}(\varphi_1),\ldots,B^H(\varphi_n)) \varphi_i.$$
It can be proved that $E\|D F\|_{\cU}^2<\infty$. This definition can be extended to space $\bD^{1,2}$ defined as the completion of the set of smooth random variables with respect to the norm:
$$\|F\|_{\bD^{1,2}}=\big(E|F|^{2} \big)^{1/2}+\big( E\|DF\|_{\cU}^2\big)^{1/2}.$$

So, the Malliavin derivative is an operator $D:\bD^{1,2} \subset L^2(\Omega) \to L^2(\Omega;\cU)$. Let $\delta:{\rm Dom} \ \delta \subset L^2(\Omega;\cU) \mapsto L^2(\Omega)$ be the adjoint of this operator. By duality,
$$E[F \delta(X)]=E[\langle DF,X\rangle_{\cU}] \quad \mbox{for all} \quad F \in \bD^{1,2}.$$

We denote $\delta (X)=\int_0^T X(t)\delta B^H(s)$. $\delta(X)$ is called the {\em divergence integral} of $X$ with respect to $B^H$.

\begin{remark}
{\rm The operator $\delta$ can also be defined for the Brownian motion, in which case it is called {\em the Skorohod integral}. This integral bears Skorohod's name since it appeared in his work \cite{skorohod75} in 1975, but in fact it had been introduced earlier by Hitsuda in a talk given at a Japan-USSR symposium in 1972 (which Skorohod attended); see \cite{hitsuda72}. The Skorohod integral coincides with the It\^o integral for {\em adapted} integrands. So, in the case of the Brownian motion, the Skorohod integral is just an extension of the It\^o integral to a set of non-adapted integrands. On the other hand, it should be emphasized that {\em there is no It\^o integral with respect to the fBm, even if the integrand is adapted}.}
\end{remark}

The operator $\delta$ {\em is not an isometry}! But we have the following useful tool for estimating the second moment of $\delta(X)$: for elements $X$ is a subspace of ${\rm Dom} \ \delta$ (denoted by $\bD^{1,2}(\cU)$),
\begin{align}
\label{iso-delta}
& E\left|\int_0^T X(s)\delta B^H(s)\right|^2 \leq E\|X\|_{\cU}^2 +E\|DX\|_{\cU \otimes \cU}^2\\
\nonumber
& \quad \quad \quad = \alpha_H E\int_0^T \int_0^T X(t)X(s)|t-s|^{2H-2}dtds+\\
\nonumber
& \quad \quad \quad \quad \quad \alpha_H^2 E \int_{[0,T]^4} D_t X(s) D_{t'}X(s')|t-t'|^{2H-2}|s-s'|^{2H-2}dtdt' dsds'.
\end{align}

The following result is the {\em It\^o formula} for the divergence integral with respect to the fBm (see Theorem 5.2.1 of \cite{nualart06}).

\begin{theorem}
\label{ito-fBm}
If $H>1/2$ and $F:[0,\infty) \times \bR \to \bR$ is a function which is continuously differentiable in $t$ and twice continuously differentiable in $x$, then
$$F(t,B_t^H)-F(0,0)=\int_0^t \frac{\partial F}{\partial t}(s,B_s^H)ds+\int_0^t \frac{\partial F}{\partial x}(s,B_s^H)\delta B^H(s)+H\int_0^t \frac{\partial^2 F}{\partial x^2}(s,B_s)s^{2H-1}ds.$$
\end{theorem}

\subsection{SDE with fractional noise}

Assume that $H>1/2$.
Consider the stochastic differential equation (SDE):
\begin{equation}
\label{SDE-fBm}
dX(t)=\sigma(X(t)) \delta B_t^{H}, \quad X(0)=0,
\end{equation}
where $\sigma$ is a globally Lipschitz function. The solution of \eqref{SDE-fBm} satisfies
$$X(t)=\int_0^t \sigma(X(s))\delta B^H(s),$$
provided that the stochastic integral is well-defined. 

Solving equation \eqref{SDE-fBm} for a general function $\sigma$ has remained an open problem for the last 20 years. To see where the difficulty comes from, let's try to set-up a Picard's iterations scheme, as in the case of the Brownian motion: $X_0(t)=0$ and for any $n \geq 0$,
$$X_{n+1}(t)=\int_0^t \sigma(X_n(s))\delta B^H(s).$$
Assume that $\sigma$ is differentiable, and $|\sigma(x)| \leq C$ and $|\sigma'(x)|\leq C'$ for all $x \in \bR$.
By \eqref{iso-delta},
\begin{align*}
E|X_{n+1}(t)|^2 & \leq \alpha_H E\int_0^t \int_{0}^t \sigma(X_n(s)) \sigma(X_n(s'))|s-s'|^{2H-2}dtdt'+\\
& \alpha_H^2\int_{[0,T]^4} D_t \big( \sigma(X_n(s)) \big) D_{t'}\big(\sigma(X_{n}(s')) \big) |t-t'|^{2H-2}|s-s'|^{2H-2} dtdt' ds ds' \\
& \leq \alpha_H C^2 E\int_0^t \int_{0}^t |X_n(s)| |X_n(s')||s-s'|^{2H-2}dtdt'+\\
& \alpha_H^2 C'{^2} E \int_{[0,T]^4} |D_t X_n(s)| |D_{t'} X_{n}(s')|  |t-t'|^{2H-2}|s-s'|^{2H-2} dtdt' ds ds',
\end{align*}
using the fact that $D_t \big(\sigma(X_n(s))\big) =\sigma'(X_n(s))D_t X_n(s)$. But it is not clear how to estimate this further.

\subsection{Particular case: SDE with multiplicative fractional noise}

Consider the particular case $\sigma(x)=x$ (but with non-vanishing initial conditions). More precisely, we look at the equation:
\begin{equation}
\label{SDE-fBm2}
dX(t)=X(t)\delta B_t^H, \quad X_0=1.
\end{equation}
The solution of this equation satisfies:
$$X(t)=1+\int_0^t X(s)\delta B^H(s),$$
provided that the stochastic integral is well-defined.

\begin{theorem}
\label{exist-sol-FBM}
Equation \eqref{SDE-fBm2} has a unique solution, the geometric fBm:
$$X(t)=\exp(B_t^H-t^{2H}/2).$$
\end{theorem}

\begin{proof}
Similarly to the Brownian motion case, there are two methods to prove this result. The first method consists in applying the It\^o formula for the divergence integral (Theorem \ref{ito-fBm}) to the function $F(t,x)=e^{x-t^{2H}/2}$. We explain the second method, which consists in writing the series expansion of the solution. As in the proof of Theorem \ref{exist-sol-BM}, it can be shown that, if it exists, the solution of \eqref{SDE-fBm2} has the following series expansion in $L^2(\Omega)$:
$$X(t)=1+\sum_{n \geq 1}\frac{1}{n!}I_n^{B^H}(t),$$
where $I_n^{B^H}(t)$ is the multiple integral of order $n$:
$$I_n^{B^H}(t)=\int_{[0,t]^n}1 dB^H(t_1) \ldots dB^H(t_n).$$
Unlike Theorem \ref{exist-sol-BM}, here it is not clear how to define an iterated integral with respect to $B^H$. But still, the multiple integrals $I_n^{B^H}(t)$ can be calculated using Hermite polynomials. Using Theorem \ref{int-Hermite} with $W=B^H$, $\cH=\cU$ and $h=1_{[0,t]}$ (with $\|h\|_{\cU}^2=t^{2H}$), we see that
$$I_n^{B^H}(t)=t^{n/2}H_n\left(\frac{B_t^H}{t^{H}} \right).$$

Hence, 
$$X(t)=1+\sum_{n\geq 1} \frac{1}{n!} I_{n}^{B^H}(t)=1+\sum_{n\geq 1}\frac{1}{n!} t^{2H} H_n\left(\frac{B_t^{H}}{t^H} \right)=\exp(B_t^H-t^{2H}/2).$$
\end{proof}

Note that for any $p>0$, $E[X(t)^p]=\exp(\frac{p(p-1)}{2} t^{2H})$ and hence the Lyapunov exponent of the solution $\{X(t)\}_{t>0}$ is infinte:
$$\lambda_p=\lim_{t \to \infty}\frac{1}{t} \log E[X(t)^p]=\infty.$$
But we can consider a {\em modified Lyapunov exponent} (or order $2H$), which is finite:
$$\lambda_p^{(H)}=\lim_{t \to \infty} \frac{1}{t^{2H}}\log E[X(t)^p]=\frac{p(p-1)}{2}.$$

\section{SPDEs with space-time homogeneous Gaussian noise}
\label{section-SPDE}

In this section, we introduce a new model for the noise perturbing an SPDE, and we give a summary of the known results for the heat and wave equations with this type of noise.

\subsection{The fractional-colored noise: a spin-off from Dalang's theory}

Taking a hint from the form \eqref{fBm-cov} of the covariance of the fBm with index $H>1/2$ and the covariance \eqref{Dalang-cov} of the spatially homogeneous Gaussian noise, we consider a zero-mean Gaussian process $F=\{F(A);A \in \cB_b(\bR_{+} \times \bR^d)\}$ with covariance:
$$E[F(A)F(B)]=\int_{(\bR_{+} \times \bR^d)^2} 1_{A}(t,x) 1_{B}(s,y)\gamma(t-s)f(x-y)dt dx ds dy=:\langle 1_{A},1_{B} \rangle_{\cH},$$
where $f$ is the Fourier transform of a tempered measure $\mu$ on $\bR^d$ and $\nu$ is the Fourier transform of a tempered measure $\nu$ on $\bR$. Using an expression of the inner product in terms of Fourier transforms, it can be shown that $\langle \cdot,\cdot \rangle_{\cH}$ is non-negative definite.

We say that $F$ is a {\em space-time homogeneous Gaussian noise}. It was introduced in \cite{BT-ALEA}, in the case $\gamma(t)=|t|^{2H-2}$ with $H \in (\frac{1}{2},1)$, where it was called ``fractional-colored'' noise.

For any $t>0$ and $A \in \cB_b(\bR^d)$, we let $F_t(A)=F([0,t] \times A)$. Then
$$E[F_t(A)F_s(B)]=R(t,s) \int_{A} \int_{B}f(x-y)dxddy,$$
where $R(t,s)=\int_0^t \int_0^s \gamma(u-v)dudv=E[Z_t Z_s]$ is the covariance of a process $(Z_t)_{t\geq 0}$ with stationary increments (with spectral measure $\nu$).
The typical examples that we are interested in are: $f(x)=|x|^{-\alpha}$ with $0<\alpha<d$
(which appeared in Dalang's work; see Remark \ref{Dalang-rem}) and
$\gamma(t)=|t|^{2H-2}$ with $1/2<H<1$ (in which case the process $(Z_t)_{t\geq 0}$ is a fBm of index $H$; see Remark \ref{Fourier-rem}).

In general, $\{F_t(A);t \geq 0,A \in \cB_b(\bR^d)\}$ may not be a martingale measure. So, in the case of random integrands, we cannot use Walsh's approach for defining a stochastic integral with respect to $F$.
But for deterministic integrands, we can use It\^o's approach. If $\varphi=1_{A}$ for some $A \in \cB_b(\bR_{+} \times \bR^d)$, we let $I^F(\varphi)=F(A)$. We extend this definition to the space $\cE(\bR_{+} \times \bR^d)$ of linear combinations of indicator functions of this form. For such functions,
$$E{I^{F}(\varphi)I^{F}(\psi)}=\int_{(\bR_{+} \times \bR^d)^2} \varphi(t,x)\psi(s,y)\gamma(t-s)f(x-y)dtdx ds dy=:\langle \varphi,\psi \rangle_{\cH}.$$
This shows that the map $\varphi \mapsto I^F(\varphi) \in L^2(\Omega)$ is an isometry, which we extend to the Hilbert space $\cH$ defined as the completion of $\cE(\bR_{+} \times \bR^d)$ with respect to the inner product $\langle \cdot,\cdot \rangle_{\cH}$.
The space $\cH$ may contain distributions. For any $\varphi \in \cH$, we denote
$$I^F(\varphi)=\int_{0}^{\infty}\int_{\bR^d}\varphi(t,x)F(dt,dx),$$
and we say that $I^F(\varphi)$ is {\em the stochastic integral} of $\varphi$ with respect to $F$.

Note that $\{F(\varphi)\}_{\varphi \in \cH}$ is an isonormal Gaussian process, corresponding to the Hilbert space $\cH$.
To define the stochastic integral in the case of random integrands, we consider the divergence operator $\delta^F$ from Malliavin calculus, which is defined as the adjoint of the Malliavin derivative, as in Section \ref{integral-fBm} above (see \cite{nualart06}). For any $X \in {\rm Dom} \ \delta^F$, we denote
$$\delta^F(X)=\int_0^{\infty}\int_{\bR^d}X(t,x)F(\delta t,\delta x).$$

\subsection{Again the linear equation}

We consider the equation
\begin{equation}
\label{SPDE-lin}
Lu(t,x)=\dot{F}(t,x) \quad t>0,x \in \bR^d,
\end{equation}
with zero initial conditions, where $L$ is a second-order partial differential operator with constant coefficients. As in Section \ref{lin-eq-white} below, let $G$ be the fundamental solution of $L$. By definition, the {\em mild solution} to \eqref{SPDE-lin} is given by:
$$u(t,x)=\int_{0}^{t}\int_{\bR^d}G(t-s,x-y)F(ds,dy),$$
provided that the stochastic integral is well-defined, i.e. $G(t-\cdot,x-\cdot)1_{[0,t]}(\cdot) \in \cH$.

The case $\gamma(t)=|t|^{2H-2}$ with $1/2<H<1$ was examined in \cite{BT-SPA}, where it was shown that:
\begin{itemize}
\item If $L=\frac{\partial}{\partial t}-\frac{1}{2}\Delta$ (heat operator), \eqref{SPDE-lin} has a random field solution if and only if
    \begin{equation}
    \label{H-cond-heat}
    \int_{\bR^d}\left(\frac{1}{1+|\xi|^2} \right)^{2H}\mu(d\xi)<\infty.
    \end{equation}
    In particular, if $f(x)=|x|^{-\alpha}$ with $\alpha \in (0,d)$, \eqref{H-cond-heat} becomes $\alpha<4H$.

\item If $L=\frac{\partial^2}{\partial t^2}-\frac{1}{2}\Delta$ (wave operator), \eqref{SPDE-lin} has a random field solution if and only if
    \begin{equation}
    \label{H-cond-wave}
    \int_{\bR^d}\left(\frac{1}{1+|\xi|^2} \right)^{H+1/2}\mu(d\xi)<\infty.
    \end{equation}
    In particular, if $f(x)=|x|^{-\alpha}$ with $\alpha \in (0,d)$, \eqref{H-cond-heat} becomes $\alpha<2H+1$.
\end{itemize}

Note that when $H=1/2$, \eqref{H-cond-heat} and \eqref{H-cond-wave} coincide with Dalang's condition \eqref{Dalang-cond}. So, formally we can say that Dalang's work \cite{dalang99} covers the case $H=1/2$. Since $H>1/2$, \eqref{H-cond-wave} is stronger than \eqref{H-cond-heat}.

\begin{remark}
{\rm
In \cite{B-JFAA}, these results were extended to a general Gaussian noise with covariance determined by some tempered measures $\nu$ on $\bR$ and $\mu$ on $\bR^d$, whose Fourier transforms are not necessarily locally integrable functions. In particularly, this covers the case when the noise behaves in time like a fBm with $H<1/2$. More precisely, in \cite{B-JFAA} it was shown in that:
\begin{itemize}
\item if $L=\frac{\partial}{\partial t}-\frac{1}{2}\Delta$ (heat operator), \eqref{SPDE-lin} has a random field solution if and only if
    $$\int_{\bR^d}\int_{\bR}\frac{1}{1+\tau^2+|\xi|^4} \nu(d\tau) \mu(d\xi)<\infty;$$

\item if $L=\frac{\partial^2}{\partial t^2}-\frac{1}{2}\Delta$ (wave operator), \eqref{SPDE-lin} has a random field solution if and only if
    $$\int_{\bR^d}\frac{1}{\sqrt{1+|\xi|^2}}\int_{\bR}\frac{1}{1+\tau^2+|\xi|^2} \nu(d\tau) \mu(d\xi)<\infty.$$
\end{itemize}
}
\end{remark}

\subsection{What we know so far about some non-linear equations}

Consider the equation
\begin{equation}
\label{SPDE-nonlin}
Lu(t,x)=u(t,x)\dot{F}(t,x) \quad t>0,x \in \bR^d,
\end{equation}
with initial condition $1$, where $L$ is the heat or wave operator on $\bR_{+} \times \bR^d$. By definition, a {\em mild Skorohod solution} to \eqref{SPDE-nonlin} satisfies:
$$u(t,x)=1+\int_0^t \int_{\bR^d}G(t-s,x-y)u(s,y)F(\delta s,\delta y),$$
provided that the stochastic integral is well-defined (as the divergence integral from Malliavin calculus). The word ``Skorohod'' is used to distinguish it from the mild Stratonovich solution, for which the stochastic integral is understood in a different sense. Here $G$ is the fundamental solution of the heat or wave operator.

We include below a summary of the known results for the heat equation:
\begin{equation}
\label{PAM}
\frac{\partial u}{\partial t}(t,x)=\frac{1}{2}\Delta u(t,x)+u(t,x)\dot{F}(t,x) \quad t>0,x \in \bR^d
\end{equation}
with initial condition $u(0,x)=1$. Equation \eqref{PAM} is called the {\em Parabolic Anderson Model} with space-time homogeneous Gaussian noise.

Items {\bf (a)}-{\bf (d)} below were obtained by the Hu-Nualart group in a series of three papers: \cite{hu-nualart09,HNS11,HHNT}. Item {\bf (e)} was proved in \cite{B-ECP}. Item {\bf (f)} is taken from the recent preprints \cite{BQS,hu-le}.

\begin{description}

\item[(a)] ({\em Existence and uniqueness}) Under Dalang's condition \eqref{Dalang-cond}, equation \eqref{PAM} has a unique mild Skorohod solution and this solution has the series expansion:
    \begin{equation}
    \label{series-exp}
    u(t,x)=1+\sum_{n\geq 1}I_n^F(f_n(\cdot,t,x))
    \end{equation}
    where $I_n^F$ is the multiple integral of order $n$ with respect to $F$ (defined as in  \cite{nualart06}, using Malliavin calculus), and the function $f_n(\cdot,t,x)$ is given by relation \eqref{def-fn} with $G$ the fundamental solution of the heat operator.

\item[(b)] ({\em Feynman-Kac (FK) representation of moments}) Under Dalang's condition \eqref{Dalang-cond}, the moments of the solution $u$ of \eqref{PAM} have the following stochastic representation:
    $$E[u(t,x)^k]=E\left[\exp \left\{ \sum_{1 \leq i<j\leq k} \int_0^t \int_0^t \gamma(r-s) f(B_r^i-B_s^j)drds\right\} \right],$$
    for any integer $k\geq 2$, where $B^1,\ldots, B^k$ are i.i.d. $d$ dimensional Brownian motions, independent of $W$. Moreover, $u(t,x)$ is the limit in $L^k(\Omega)$ (hence in probability) of a sequence of non-negative random variables, and therefore $u(t,x)\geq 0$ a.s.

\item[(c)] ({\em Intermittency}) If $\gamma(t)=|t|^{2H-2}$ for some $H \in (\frac{1}{2},1)$ and $f(x)=|x|^{-\alpha}$ for some $\alpha \in (0,d \wedge 2)$, then for any $p>0$,
    $$\exp\big(c_1 p^{\frac{4-\alpha}{2-\alpha}} t^{\rho}\big) \leq E[u(t,x)^p] \leq \exp\big(c_2 p^{\frac{4-\alpha}{2-\alpha}} t^{\rho}\big),$$
    where $c_1>0,c_2>0$ are constants independent of $t,x$ and $p$, and
    $$\rho=\frac{4H-\alpha}{2-\alpha}.$$

\item[(d)] ({\em FK representation of the solution}) If $\gamma(t) \leq c|t|^{2H-2}$ for all $t \in \bR$, for some $H \in (\frac{1}{2},1)$ and $c>0$, and
    \begin{equation}
    \label{FK-cond}
    \int_{\bR^d}\left( \frac{1}{1+|\xi|^2}\right)^{2H-1}\mu(d\xi)<\infty,
    \end{equation}
    then
   $$u(t,x)=E\Big[\exp\Big(W(A_{t,x}^B)-\frac{1}{2}\|A_{t,x}^B\|_{\cH}^2 \Big) \Big],$$
   where $A_{t,x}^B $ is a random variable in $L^2(\Omega;\cH)$, which is measurable with respect to a $d$-dimensional Brownian motion $B=(B_t)_{t\geq 0}$, starting at $x$ and independent of $W$.

\item[(e)] ({\em Another FK representation of the second moment}) Under Dalang's condition \eqref{Dalang-cond},
    $$E[u(t,x)^2]=1+e^{t^2} \sum_{n\geq 1} \, \, \sum_{i_1,\ldots,i_n \ distinct} E\Big[\prod_{j=1}^n \gamma(T_{i_j}-S_{i_j}) \prod_{j=1}^{n}f(B_{T_{i_j}}^1 -B_{S_{i_J}}^2)1_{A_{i_1,\ldots,i_n}(t)} \Big],$$
    where $\{P_i=(T_i,S_i), i\geq 1\}$ are the points of a Poisson random measure $N$ on $\bR_{+}^2$ with intensity given by the Lebesque measure, $A_{i_1,\ldots,i_n}(t)$ is the event that $N$ has points $P_{i_1},\ldots,P_{i_n}$ in the set $[0,t]^2$, and $B^1,B^2$ are independent Brownian motions on $\bR^d$ starting at $x$, which are independent of $N$.

\item[(e)] ({\em H\"older continuity}) If
\begin{equation}
\label{Holder-cond}
\int_{\bR^d} \left(\frac{1}{1+|\xi|^2} \right)^{\eta}\mu(d\xi)<\infty
\end{equation}
for some $\eta \in (0,1)$, then for any $p \geq 2$, $t,t' \in [0,T]$ and $x,x' \in \bR^d$,
$$\|u(t,x)-u(t',x')\|_p \leq C_T\Big(|t-t'|^{\frac{1-\eta}{2}}+|x-x'|^{1-\eta} \Big),$$
where $\|\cdot\|_p$ is the norm in $L^p(\Omega)$ and $C_T>0$ is a constant depending on $T$. By Kolmogorov's criterion, $u$ has a H\"older continuous modification of order $\frac{1-\eta}{2}-\varepsilon$ in time and order $1-\eta-\varepsilon$ in space, for any $\varepsilon>0$.
In particular, if $f(x)=|x|^{-\alpha}$ for some $\alpha \in (0,d \wedge 2)$, condition \eqref{Holder-cond} holds for any $\eta \in (\frac{\alpha}{2},1)$ and $u$ has a H\"older continuous modification of order $\frac{1}{2}(1-\frac{\alpha}{2})-\varepsilon$ in time and order $1-\frac{\alpha}{2}-\varepsilon$ in space. If in addition, $\gamma(t)=|t|^{2H-2}$ for some $H \in (\frac{1}{2},1)$, then $u$ has a H\"older continuous modification of order $\frac{1}{2}(2H-\frac{\alpha}{2})-\varepsilon$ in time and order $2H-\frac{\alpha}{2}-\varepsilon$ in space.

\end{description}

We consider now the case of the wave equation:
\begin{equation}
\label{HAM}
\frac{\partial^2 u}{\partial t^2}(t,x)=\Delta u(t,x)+u(t,x)\dot{F}(t,x) \quad t>0,x \in \bR^d
\end{equation}
with initial condition $u(0,x)=1$ and $\frac{\partial u}{\partial t}(0,x)=0$. Equation \eqref{HAM} is called the {\em Hyperbolic Anderson Model}. There are fewer results for this equation in the literature. We list these results below. Items {\bf [a,d]} are taken from \cite{balan-song}, and items {\bf [b,c]} from \cite{BC-AOP}.

\begin{description}

\item[(a)] ({\em Existence and Uniqueness}) Under Dalang's condition \eqref{Dalang-cond}, equation \eqref{HAM} has a unique mild Skorohod solution for any $d\geq 1$, and this solution has the same series expansion \eqref{series-exp} as in the case of the Parabolic Anderson Model, but with $G$ replaced by the fundamental solution of the wave operator.

\item[(b)] ({\em FK representation of the second moment}) Under Dalang's condition \eqref{Dalang-cond}, if $d\leq 3$,
    \begin{align*}
    E[u(t,x)^2]& =1+\sum_{n\geq 1}\, \, \sum_{i_1,\ldots,i_n \ distinct} E\left[ \prod_{j=1}^{n}\gamma(T_{i_j}-S_{i_j})\prod_{j=1}^{n}f(X_{T_{i_j}}^1-X_{S_{i_j}}^2)
     \right. \\
    & \qquad \qquad \qquad \qquad \qquad \qquad  \left. \prod_{j=1}^{n}(\tau_j-\tau_{j-1}) \prod_{j=1}^n (\tau_j'-\tau_{j-1}') 1_{A_{i_1 \ldots i_n}(t)}\right],
    \end{align*}
    where $\{P_i=(T_i,S_i), i\geq 1\}$ are the points of a Poisson random measure $N$ on $\bR_{+}^2$ with intensity given by the Lebesque measure, $A_{i_1,\ldots,i_n}(t)$ is the event that $N$ has points $P_{i_1},\ldots,P_{i_n}$ in the set $[0,t]^2$, $\tau_1<\ldots < \tau_n$ and $\tau_1'<\ldots < \tau_n'$ are the points $T_{i_1},\ldots,T_{i_n}$, respectively $S_{i_1},\ldots,S_{i_n}$ arranged in increasing order. In this representation, the processes $X^1=(X^1_s)_{s \in [0,t]}$ and $X^2=(X_s^2)_{s \in [0,t]}$ are defined as follows: $X_0^1=X_0^2=x$ and for any $i=0,\ldots,n$,
    \begin{align*}
    X_{s}^1 &=X_{\tau_i}^1+(t-\tau_i)\Theta_i^1 \quad \mbox{if $\tau_i\leq s<\tau_{i+1}$}\\
    X_{s}^2 &=X_{\tau_i'}^2+(t-\tau_i')\Theta_i^2 \quad \mbox{if $\tau_i'\leq s<\tau_{i+1}'$},
    \end{align*}
    where $\tau_0=0,\tau_{n+1}=t$, and $\Theta_i^1,\Theta_i^2,i\geq 1$ are i.i.d. random variables with values in $\bR^d$, independent of $N$ and with density function (or distribution, if $d=3$) $G(1,\cdot)$.

\item[(c)] ({\em Intermittency}) Suppose that $\gamma(t)=|t|^{2H-2}$ for some $H \in (\frac{1}{2},1)$ and $f(x)=|x|^{-\alpha}$ for some $\alpha \in (0,d \wedge 2)$. Then for any $p\geq 2$,
    $$E|u(t,x)|^p \leq C_0^p \exp \Big(C_1 p^{\frac{4-\alpha}{3-\alpha}} t^{\rho} \Big) \quad \mbox{and} \quad E|u(t,x)|^2 \geq C_2 \exp \Big(C_3 p^{\frac{4-\alpha}{3-\alpha}} t^{\rho} \Big),$$
  where $C_0, \ldots,C_3$ are positive constants independent of $t,x$ and $p$, and
   $$\rho=\frac{2H+2-\alpha}{3-\alpha}.$$
   (The lower bound for $E|u(t,x)|^p$ with $p>2$ is an open problem.)

\item[(d)] ({\em H\"older continuity}) If \eqref{Holder-cond} holds,
then for any $p \geq 2$, $t,t' \in [0,T]$ and $x,x' \in \bR^d$,
$$\|u(t,x)-u(t',x')\|_p \leq C_T\Big(|t-t'|^{1-\eta}+|x-x'|^{1-\eta} \Big),$$
where $\|\cdot\|_p$ is the norm in $L^p(\Omega)$ and $C_T>0$ is a constant depending on $T$. By Kolmogorov's criterion, $u$ has a H\"older continuous modification of order $1-\eta-\varepsilon$ in time and space, for any $\varepsilon>0$.
In particular, if $f(x)=|x|^{-\alpha}$ for some $\alpha \in (0,d \wedge 2)$, condition \eqref{Holder-cond} holds for any $\eta \in (\frac{\alpha}{2},1)$ and $u$ has a H\"older continuous modification of order $1-\frac{\alpha}{2}-\varepsilon$ in time and space.

\end{description}

{\bf Acknowledgement:} The author would like to thank Vojkan Jaksic for the invitation to give these lectures, Armen Shirikyan and Vahagn Nersesyan for useful discussions, and CRM for the hospitality.

\end{document}